\documentclass[11pt,reqno]{amsart}
\usepackage{amssymb}
\usepackage{upgreek}
\usepackage{dsfont}
\usepackage{mathrsfs}
\usepackage[all]{xy}
\newtheorem{thm}{Theorem}[section]
\newtheorem{lem}[thm]{Lemma}
\newtheorem{prop}[thm]{Proposition}

\theoremstyle{definition}
\newtheorem{dfn}[thm]{Definition}
\newtheorem{ex}[thm]{Example}

\theoremstyle{remark}
\newtheorem{remark}[thm]{Remark}

\newtheorem{notation}[thm]{Notation}

\newtheorem{assumption}[thm]{Assumption}

\newcommand{\CA}{{\mathcal{A}}}

\newcommand{\bE}{{\overline{\mathcal{E}}}}

\newcommand{\CL}{{\mathcal{L}}}
\newcommand{\CB}{{\mathcal{B}}}

\newcommand{\af}{\alpha}
\newcommand{\bt}{\beta}
\newcommand{\gm}{\gamma}

\newcommand{\ld}{\lambda}

\newcommand{\Q}{{\mathbb{Q}}}
\newcommand{\Z}{{\mathbb{Z}}}

\newcommand{\C}{{\mathbb{C}}}
\newcommand{\N}{{\mathbb{N}}}
\newcommand{\T}{{\mathbb{T}}}

\usepackage[pdftex]{hyperref}

\begin{document}


\title[Trace on the Thue--Morse labeled graph $C^*$-algebra]
{Unique tracial state on the labeled graph $C^*$-algebra associated to The Thue--Morse sequence}

\author[S. H. Kim]{Sun Ho Kim}
\address{
BK21 Plus Mathematical Sciences Division\\
Seoul National University\\
Seoul, 151--747\\
Korea} \email{sunho.kim.math\-@\-gmail.\-com }

\thanks{Research supported by BK21 PLUS SNU Mathematical Sciences Division}



\begin{abstract}
We give a concrete formula for the unique faithful trace on the finite simple non-AF labeled graph $C^*$-algebra $C^*(E_{\mathbb{Z}}, \mathcal{L}, \overline{\mathcal{E}}_{\mathbb{Z}})$ associated to the Thue--Morse sequence $(E_{\mathbb{Z}}, \mathcal{L})$. 
Our result provides an alternative proof of the existence of a labeled graph $C^*$-algebra that is not Morita equivalent to any graph $C^*$-algebras.
Furthermore, we compute the $K$-groups of $C^*(E_{\mathbb{Z}}, \mathcal{L}, \overline{\mathcal{E}}_{\mathbb{Z}})$ using the path structure of the Thue--Morse sequence. 
\end{abstract}

\maketitle

\setcounter{equation}{0}

\section{Introduction}

\noindent In order to study the relationship between the class of $C^*$-algebras and the class of dynamical systems, Cuntz and Krieger \cite{CK} introduced $C^*$-algebras associated to shifts of finite type.
Since then during the past thirty years there have been many generalizations of the Cuntz-Krieger algebras via a lot of different approaches.
These generalizations include, for example, graph $C^*$-algebras $C^*(E)$ of directed graphs $E$ (see ~\cite{BPRS,FLR, KPR, KPRR} among others), Exel-Laca algebras $\mathcal{O}_A$ of infinite $\{0, 1\}$-matrices $A$ ~\cite{EL}, ultragraph $C^*$-algebras $C^*(\mathcal{G})$ ~\cite{To1}, and Matsumoto algebras $\mathcal{O}_{\Lambda}$, $\mathcal{O}_{\Lambda^*}$ of one-sided shift spaces $\Lambda$ over finite alphabets ~\cite{Ma}.
To unify aforementioned algebras, Bates and Pask ~\cite{BP1} introduced a class of $C^*$-algebras $C^*(E, \CL, \bE)$ associated to labeled spaces $(E, \CL, \bE)$, which we call the labeled graph $C^*$-algebras.

It is now well known that simple graph $C^*$-algebras are classifiable by their K-theories (see ~\cite[Theorem 3.2]{RS} and ~\cite[Remark 4.3]{R}), and are either AF or purely infinite (see ~\cite[Corollary 3.10]{KPR}).
Bates and Pask showed in ~\cite{BP2} that there exists a unital simple purely infinite labeled graph $C^*$-algebra whose  $K_0$-group is not finitely generated.  
Since the $K_0$-group of a unital graph $C^*$-algebra is always finitely generated, this example shows that the class of labeled graph $C^*$-algebras is strictly larger than the class of graph $C^*$-algebras up to isomorphism.
On the other hand, the three classes of graph $C^*$-algebras, Exel-Laca algebras, and ultragraph $C^*$-algebras are known to be equal up to Morita equivalence ~\cite{KMST2}, and so the following question naturally arises: 
``Up to Morita equivalence, is the class of labeled graph $C^*$-algebras still larger than the class of graph $C^*$-algebras?"
A negative answer to the question is given very recently in ~\cite{JKKP} by showing that there exist unital simple finite labeled graph $C^*$-algebras which are neither AF nor purely infinite.
(Note here that the property of being AF or purely infinite, simple is preserved under Morita equivalence.)

To show the existence of such a finite simple unital non-AF labeled graph $C^*$-algebra, it is proven in ~\cite{JKKP} that there exists a family of non-AF simple unital labeled graph $C^*$-algebras with traces, more precisely, the algebras in the family are obtained as crossed products of Cantor minimal subshifts.
If a Cantor minimal subshift is uniquely ergodic (for example, the subshift of Thue--Morse sequence is uniquely ergodic), its corresponding labeled graph $C^*$-algebra has a unique trace. 
In this paper, we provide another proof for the existence and uniqueness of a trace on the labeled graph $C^*$-algebra of the Thue--Morse sequence, and we give a concrete formula of the trace.
We also obtain $K$-groups of this labeled graph $C^*$-algebra associated to the Thue--Morse sequence.

The paper is organized as follows. 
In Section ~\ref{Preliminaries}, we review the definitions of a labeled graph $C^*$-algebra and the Thue--Morse sequence  $\omega$.
We also describe the structure of the AF-core of the labeled graph $C^*$-algebra $C^*(E_{\Z}, \CL, \bE_{\Z})$ associated to the Thue--Morse sequence. 
Then in Section ~\ref{Labeled paths on Thue--Morse sequence}, we discuss properties of labeled paths of $(E_{\Z}, \CL)$ by using a new notation which helps us to avoid complicated computations.
In Section ~\ref{4} we prove the existence of the unique trace on $C^*(E_{\Z}, \CL, \bE_{\Z})$ and give a concrete formula (see Theorem ~\ref{main} and Proposition ~\ref{unique trace}).
Finally, in Section ~\ref{5} we give a computation of $K$-groups and in Section ~\ref{6} a representation on $\ell^2(\Z)$ of $C^*(E_{\Z}, \CL, \bE_{\Z})$.

\section{Preliminaries}\label{Preliminaries}

\subsection{\bf Labeled graph $C^*$-algebras} 
We follow the notational convention of \cite{KPR} for graphs and of \cite{BCP, BP2} for labeled graphs and their $C^*$-algebras.

A (directed) {\it graph} $E = (E^0, E^1, r, s)$ consists of a vertex set $E^0$, an edge set $E^1$, range and source maps $r, s : E^1 \to E^0$. 
A {\it path} $\lambda$ is a sequence of edges $\lambda_1 \lambda_2 \dots \lambda_n$ with $r(\lambda_i) = s(\lambda_{i+1})$ for $i = 1, \dots, n-1$.
The {\it length} of a path $\lambda = \lambda_1 \lambda_2 \dots \lambda_n$ is $|\lambda| := n$ and we denote by $E^n$ the set of all paths of length $n$.
The range and source maps can be naturally extended to $E^n$ by $r(\lambda) := r(\lambda_n)$, $s(\lambda) := s(\lambda_1)$.   
We also denote by $E^* = \cup_{n \geq 0} E^n$ the set of all vertices and all finite paths in $E$.  
A vertex is called a {\it source} if it receives no edges and a {\it sink} if it emits no edges.

\begin{assumption}
Throughout the paper, we assume that a graph has no sinks or sources.  
\end{assumption}

Let $\CA$ be an alphabet set. A {\it labeling map} $\CL$ is a map from $E^1$ onto $\CA$. 
For a path $\lambda= \lambda_1 \lambda_2 \dots \lambda_n \in E^* \setminus E^0$, we define it labeling $\CL(\lambda)$ to be $\CL(\lambda_1) \dots \CL(\lambda_n)$ and denote by $\CL^*(E)$ the set $\cup_{n \geq 1} \CL(E^n)$ of all labeled paths.
Given a graph $E$ and a labeling map $\CL$, we call $(E, \CL)$ a {\it labeled graph}.
The {\it range} of a labeled path $\alpha$ is defined by
$$
r(\alpha) := \{r(\lambda) \in E^0 : \CL(\lambda) = \alpha\},
$$ 
and the {\it relative range} of $\alpha$ with respect to $A \subset E^0$ is defined by
$$
r(A, \alpha) := \{r(\lambda) \in E^0 : s(\lambda) \in A  \mbox{ and } \CL(\lambda) = \alpha\}.
$$ 

Let $\CB$ be a subset of $2^{E^0}$. 
If $\CB$ contains all ranges of labeled paths, and is closed under finite union, finite intersection, and relative range, then we call $\CB$ an {\it accommodating set} for $(E, \CL)$ and a triple $(E, \CL, \CB)$ a {\it labeled space}.

For $A,B \subset E^0$ and $1 \leq n \leq \infty$, let
 $$ AE^n =\{\ld\in E^n\,:\, s(\ld)\in A\},\ \
  E^nB=\{\ld\in E^n\,:\, r(\ld)\in B\}.$$
A labeled space $(E,\CL,\CB)$ is said to be {\it set-finite}
 ({\it receiver set-finite}, respectively) if for every $A\in \CB$ and $1 \leq l < \infty$ 
 the set  $\CL(AE^l)$ ($\CL(E^lA)$, respectively) is finite. 

If a labeled space $(E, \CL, \CB)$ satisfies
$$
r(A, \alpha) \cap r(B, \alpha) = r(A \cap B, \alpha)
$$
for all $A, B \in \CB$, then $(E, \CL, \CB)$ is called {\it weakly left-resolving}.
We will only consider the smallest accommodating set $\bE$ which is {\it non-degenerate} (namely closed under relative complement).
We also assume that $(E, \CL, \bE)$ always is set-finite, receiver set-finite and weakly left-resolving.  

We define an equivalence relation $\sim_l$ on $E^0$; we write $v\sim_l w$ if $\CL(E^{\leq l} v)=\CL(E^{\leq l} w)$ as in \cite{BP2}.
The equivalence class $[v]_l$ of $v$ is called a {\it generalized vertex}.
If $k>l$,  $[v]_k\subset [v]_l$ is obvious and $[v]_l=\cup_{i=1}^m [v_i]_{l+1}$ for some vertices  $v_1, \dots, v_m\in [v]_l$. 
Moreover every set $A\in \bE$ is the finite union of generalized vertices, that is, 
\begin{eqnarray}\label{generalizedvertex} 
A = \cup_{i=1}^{n}  [v_i]_l 
\end{eqnarray}
for some $v_i  \in E^0$, $l\geq 1$, and $n\geq 1$ (see \cite[Remark 2.1 and Proposition 2.4.(ii)]{BP2} and \cite[Proposition 2.3]{JKK}).

\begin{dfn}{\rm (\cite[Definition 2.1]{BCP})}
Let $(E, \CL, \bE)$ be a labeled space. 
We define a {\it representation} of a labeled space $(E, \CL, \bE)$ to be a set consisting of partial isometries $\{s_a : a \in \CA\}$ and projections $\{p_A : A \in \bE \}$ such that for $a, b \in \CA$ and $A, B \in \bE$,
\begin{enumerate}
\item[(i)]  $p_{\emptyset}=0$, $p_Ap_B=p_{A\cap B}$, and
$p_{A\cup B}=p_A+p_B-p_{A\cap B}$,
\item[(ii)] $p_A s_a=s_a p_{r(A,a)}$,
\item[(iii)] $s_a^*s_a=p_{r(a)}$ and $s_a^* s_b=0$ if $a \neq b$,
\item[(iv)] for each $A \in \bE$,
$$p_A=\sum_{a \in\CL(AE^1)} s_a p_{r(A,a)}s_a^*.$$
\end{enumerate}
It was proven in ~\cite[Theorem 4.5]{BP1} that there always exists a universal representation for a labeled space.
We call a $C^*$-algebra generated by a universal representation of $(E, \CL, \bE)$ a {\it labeled graph $C^*$-algebra} $C^*(E, \CL, \bE)$. 
\end{dfn}

\begin{remark}\label{review remarks}
Let $(E,\CL,\bE)$ be a  labeled space.
\begin{enumerate}
\item[(i)] 
By \cite[Lemma 4.4]{BP1},
it follows that  
\begin{align*}
C^*(E,\CL,\bE) & =\overline{\rm span}\{s_\af p_A s_\bt^*\,:\, \af,\,\bt\in  \CL^*(E),\ A\in \bE\} \\
                         & =\overline{\rm span}\{s_\af p_{r(\mu)} s_\bt^*\,:\, \af,\,\bt,\, \mu \in  \CL^*(E)\}.
\end{align*}
Moreover, from $\displaystyle r(\mu) = \cup_{a \in \CA} \, r(a\mu)$, we may assume that $|\mu| > |\alpha|, |\beta|$.

\item[(ii)] Universal property of  $C^*(E,\CL,\bE)=C^*(s_a, p_A)$ 
defines a strongly continuous action
$\gm:\mathbb T\to {\rm Aut}(C^*(E,\CL,\bE))$,  called
the {\it gauge action}, given by
$$\gm_z(s_a)=zs_a \ \text{ and } \  \gm_z(p_A)=p_A$$ 
for $a \in \CA$ and $A \in \bE$.

\item[(iii)]  
Let $C^*(E,\CL,\bE)^\gm$ denote the fixed point algebra of the 
gauge action. It is well known that $C^*(E,\CL,\bE)^\gm$ is an AF algebra and 
$$
C^*(E,\CL,\bE)^\gm=\overline{\rm span}\{s_\af p_{r(\mu)} s_\bt^*: |\af|=|\bt|,\ \alpha, \beta, \mu \in \CL^*(E)\}.
$$
Moreover, since $\mathbb T$ is a compact group, 
there exists a faithful conditional expectation 
$$\Psi: C^*(E,\CL,\bE)\to C^*(E,\CL,\bE)^\gm.$$

\end{enumerate} 
\end{remark}

\subsection{The Thue--Morse sequence} 

We interchangeably use the terms `block' or `word' with the term `path'.
Given an infinite path $x = \dots x_{-2} x_{-1} . x_0 x_1 x_2\dots $ or a finite path $x = x_1 x_2 \dots x_{|x|}$, we write $x_{[m,n]}$ for the block $x_m x_{m+1} \dots x_n$ for $m < n$. 
Recall that the (two-sided) Thue--Morse sequence 
$$\omega=\cdots \omega_{-2} \omega_{-1} . \omega_0 \omega_1 \omega_2 \cdots $$ 
is defined by 
$\omega_0=0$, $\omega_1=\bar 0:=1$ $(\bar 1:=0)$, 
$\omega_{[0,3]}:=\omega_0 \omega_1 \bar \omega_0 \bar \omega_1 =0110$, 
$\omega_{[0,7]}:=\omega_0 \omega_1 \omega_2 \omega_3\bar \omega_0 \bar \omega_1 \bar \omega_2 \bar \omega_3 = 01101001$ 
and so on, and then $\omega_{-i} := \omega_{i-1}$ for $i\geq 1$. 
Let $(E_{\Z}, \CL)$ be the following labeled graph labeled by the Thue--Morse sequence:

\vskip 1pc
\hskip 2pc 
\xy /r0.3pc/:(-44.2,0)*+{\cdots};(44.3,0)*+{\cdots};
(-40,0)*+{\bullet}="V-4";
(-30,0)*+{\bullet}="V-3";
(-20,0)*+{\bullet}="V-2";
(-10,0)*+{\bullet}="V-1"; (0,0)*+{\bullet}="V0";
(10,0)*+{\bullet}="V1"; (20,0)*+{\bullet}="V2";
(30,0)*+{\bullet}="V3";
(40,0)*+{\bullet}="V4";
 "V-4";"V-3"**\crv{(-40,0)&(-30,0)};
 ?>*\dir{>}\POS?(.5)*+!D{};
 "V-3";"V-2"**\crv{(-30,0)&(-20,0)};
 ?>*\dir{>}\POS?(.5)*+!D{};
 "V-2";"V-1"**\crv{(-20,0)&(-10,0)};
 ?>*\dir{>}\POS?(.5)*+!D{};
 "V-1";"V0"**\crv{(-10,0)&(0,0)};
 ?>*\dir{>}\POS?(.5)*+!D{};
 "V0";"V1"**\crv{(0,0)&(10,0)};
 ?>*\dir{>}\POS?(.5)*+!D{};
 "V1";"V2"**\crv{(10,0)&(20,0)};
 ?>*\dir{>}\POS?(.5)*+!D{};
 "V2";"V3"**\crv{(20,0)&(30,0)};
 ?>*\dir{>}\POS?(.5)*+!D{};
 "V3";"V4"**\crv{(30,0)&(40,0)};
 ?>*\dir{>}\POS?(.5)*+!D{};
 (-35,1.5)*+{0};(-25,1.5)*+{1};
 (-15,1.5)*+{1};(-5,1.5)*+{0};(5,1.5)*+{\omega_0=0};
 (15,1.5)*+{\omega_1=1};(25,1.5)*+{1};(35,1.5)*+{0};
 (0.1,-2.5)*+{v_0};(10.1,-2.5)*+{v_1};
 (-9.9,-2.5)*+{v_{-1}};
 (-19.9,-2.5)*+{v_{-2}};
 (-29.9,-2.5)*+{v_{-3}};
 (-39.9,-2.5)*+{v_{-4}}; 
 (20.1,-2.5)*+{v_{2}};
 (30.1,-2.5)*+{v_{3}};
 (40.1,-2.5)*+{v_{4}}; 
\endxy 

\vskip 1pc
\noindent 
For a path $\af=\af_1\dots \af_{|\af|}\in \CL^*(E_{\Z})$, we also write
$$\af^{-1}:=  \af_{|\af|}\dots \af_1 \ \text{ and }\ 
 \overline{\af}:=\overline{\af_1}\dots \overline{\af_{|\af|}}.$$ 
Then clearly $(\af^{-1})^{-1} = \af$, $\overline{\af^{-1}} = ({\overline{\af}})^{-1}$, 
and $(\af\bt)^{-1}=\bt^{-1}\af^{-1}$ for $\af,\bt\in \CL^*(E_{\Z})$.

\subsection{\bf Labeled graph $C^*$-algebra associated to the Thue--Morse sequence} 

From now on, we mainly consider the labeled space $(E_{\Z}, \CL, \bE_{\Z})$ and the labeled graph $C^*$-algebra $C^*(E_{\Z}, \CL, \bE_{\Z})$ associated to the Thue--Morse sequence.
Clearly the $C^*$-algebra $C^*(s_a, p_A) := C^*(E,\CL,\bE)$ is unital with the unit $s_0s_0^*+s_1s_1^*$ and is known to be a simple $C^*$-algebra which is finite but not AF (see \cite[Example 4.9]{JKK} and \cite[Theorem 3.7]{JKKP}). 

It is useful to note that 
\begin{eqnarray}\label{accommodating set form} 
\bE_{\Z}=\big\{ \bigsqcup_{ \substack{1\leq i\leq K\\ |\af_i|=m}} r(\af_i): \, m\geq 1,\ K\geq 1\, \big\}.
\end{eqnarray}
Actually since the graph $E_{\Z}$ consists of a single bi-infinite path, it is easy to see that for each $v \in E^0$, $\CL(E^l v) =\{\alpha\}$ is a singleton set and so $[v]_l = r(\alpha)$.
Equivalently,
$$
r(\alpha) = [v]_{|\alpha|} \,\, \mbox{ for any } v \in r(\alpha).
$$
Thus every set $A \in \bE_{\Z}$ is a disjoint union of $r(\alpha_i)$ for $i = 1, \dots, n$ by (\ref{generalizedvertex}).
Moreover, we can assume that the lengths $|\alpha_i|$ are all equal since $r(\alpha) = r(0\alpha) \sqcup r(1\alpha)$.
Hence (\ref{accommodating set form}) follows. 

Furthermore for $\alpha, \beta \in \CL^*(E_{\Z})$ with $|\alpha|=|\beta|$, the intersection of two range sets $r(\alpha) \cap r(\beta)$ is nonempty if and only if $\alpha = \beta$.
That is, $s_{\alpha} p_{r(\mu)} s_{\beta}^* \neq 0$ for $|\alpha| = |\beta| < |\mu|$ if and only if $\alpha = \beta$ and $\emptyset \neq r(\mu)  \subset r(\alpha)$. 
Note that $r(\mu) \subset r(\alpha)$ if and only if $\mu = \beta \alpha$ for some $\beta \in \CL^*(E_{\Z})$.
From this observation, we have
\begin{align*}
C^*(E_{\Z}, \CL, \bE_{\Z})^{\gamma} & = \overline{\rm{span}} \{s_{\alpha} p_{r(\mu)} s_{\beta}^* : |\alpha| = |\beta|, \alpha, \beta, \mu \in \CL^*(E_{\Z}) \}\\
                                                & = \overline{\rm{span}} \{s_{\alpha} p_{r(\mu)} s_{\alpha}^* : \alpha, \mu \in \CL^*(E_{\Z}) \}\\
                                                & = \overline{\rm{span}} \{s_{\alpha} p_{r(\beta \alpha)} s_{\alpha}^* : \alpha, \beta \in \CL^*(E_{\Z}) \}.
\end{align*}

\begin{remark}\label{remark_AF}
We will use in the Section ~\ref{4} the following structural properties of the AF algebra $C^*(E_{\Z}, \CL, \bE_{\Z})^{\gamma}$.
\begin{enumerate}
\item[(i)] For each $k \geq 1$, let
$$
F_k = \overline{\rm{span}} \{ s_{\alpha} p_{r(\beta \alpha)} s_{\alpha} : |\alpha| = |\beta| = k \},
$$
be the finite dimensional $C^*$-subalgebra of $C^*(E_{\Z}, \CL, \bE_{\Z})^{\gamma}$ and let $\iota_k : F_k \to F_{k+1}$ be the inclusion map:
\begin{equation}\label{iota_inclusion}
\iota_k (s_{\alpha} p_{r(\beta \alpha)} s_{\alpha}) = \sum_{a \in \{0,1\}} s_{\alpha a} p_{r(\beta \alpha a)} s_{\alpha a} 
                                                                              = \sum_{a,b \in \{0, 1\}} s_{\alpha a} p_{r(b \beta \alpha a)} s_{\alpha a}.
\end{equation}
Then the AF algebra $C^*(E_{\Z}, \CL, \bE_{\Z})^{\gamma}$ is the inductive limit, $\displaystyle\lim_{\longrightarrow} (F_k, \iota_k)$.

\item[(ii)] For each $k \geq 1$, the $C^*$-subalgebra $F_k$ of $C^*(E_{\Z}, \CL, \bE_{\Z})^{\gamma}$ is commutative and isomorphic to $\mathbb{C}^{d_k}$ where $d_k = |\CL(E^{2k})|$. 
This directly follows from the fact that the product of two elements $s_{\alpha} p_{r(\beta \alpha)} s_{\alpha}^*$, $s_{\alpha'} p_{r(\beta' \alpha')} s_{\alpha'}^* \in F_k$ is equal to
\begin{align*}
(s_{\alpha} p_{r(\beta \alpha)} s_{\alpha}^*) (s_{\alpha'} p_{r(\beta' \alpha')} s_{\alpha'}^*) & = \delta_{\alpha, \alpha'} s_{\alpha} p_{r(\beta \alpha) \cap r(\beta' \alpha')} s_{\alpha'}^*\\
& =\delta_{\beta \alpha, \beta' \alpha'} s_{\alpha} p_{r(\beta \alpha)} s_{\alpha}^*.
\end{align*}

\end{enumerate}
\end{remark}

\section{Labeled paths on the Thue--Morse sequence}\label{Labeled paths on Thue--Morse sequence}

\noindent For $b, c \in \CL^*(E_{\Z})$ with $|c| = n$, the {\it product} $b \times c$ denotes the block (of length $|b| \times |c|$) obtained by
$n$ copies of $b$ or $\overline{b}$ according to the rule: choosing the $i$th copy as $b$ if $c_i = 0$ and $\overline{b}$ if $c_i = 1$ (see \cite{Ke}). 
For example, if $b = 01$ and $c = 011$, then $b \times c$ is equal to $b \overline{b} \overline{b} = 011010$. 
Then the recurrent sequence
$$
01 \times 01 \times \cdots
$$
is a one-sided Morse sequence.

To compute the trace values on $C^*(E_{\Z}, \CL, \bE_{\Z})$, we will use the following notation for convenience. 
Let 
$0^{(0)} := 0$, $1^{(0)} := 1$
and
$$0^{(n)} := 0^{(n-1)} 1^{(n-1)}, \quad 1^{(n)} := 1^{(n-1)} 0^{(n-1)}$$
for all $n$, inductively. 
Then the length of $i^{(n)}$ ($i =0$ or $1$) is $2^n$, and
$$
0^{(n)} = \underbrace{01 \times \cdots \times 01}_{(n-1)\mbox{-times}} \times 01 \mbox{ and } 1^{(n)} = \underbrace{01 \times \cdots \times 01}_{(n-1)\mbox{-times}} \times 10.
$$
Note that
$\overline{i^{(n)}} = \overline{i}^{(n)}= (1-i)^{(n)}$
and
$$(i^{(n)})^{-1} = \left\{ \begin{array}{ll} \overline{i}^{(n)}, & \hbox{if n is odd,} \\
                                                                   i^{(n)}, & \hbox{if n is even}
                                         \end{array}
                                \right.
$$
for $i = 0, 1$.
Using the notation, we see that this Morse sequence has a fractal aspect:
\begin{align} \label{defofmorse}
\omega & \ =  \cdots \omega_{-4} \omega_{-3} \omega_{-2} \omega_{-1} . \omega_0 \omega_1 \omega_2 \omega_3 \cdots \nonumber \\ 
           & \ = \cdots 0110.0110 \cdots \nonumber \\
           & \ =  \cdots 1^{(1)} 0^{(1)} 0^{(1)} 1^{(1)}. 0^{(1)} 1^{(1)} 1^{(1)} 0^{(1)} \cdots \nonumber\\
           & \ =  \cdots 0^{(2)} 1^{(2)} 1^{(2)} 0^{(2)}. 0^{(2)} 1^{(2)} 1^{(2)} 0^{(2)} \cdots \nonumber\\
           & \quad \quad \quad \quad \vdots \nonumber\\
           & \ = \left\{ \begin{array}{ll}  \cdots 1^{(n)} 0^{(n)} 0^{(n)} 1^{(n)}. 0^{(n)} 1^{(n)} 1^{(n)} 0^{(n)} \cdots,  \hbox{ if n is odd,}\\
                \cdots 0^{(n)} 1^{(n)} 1^{(n)} 0^{(n)}. 0^{(n)} 1^{(n)} 1^{(n)} 0^{(n)} \cdots, \hbox{ if n is even.} \end{array}
                                \right.
\end{align}

\begin{lem}\label{anotherpath}
Let $\alpha$ be a labeled path in $\CL^*(E_{\Z})$. 
Then both $\alpha^{-1}$ and $\overline{\alpha}$ appear in $\CL^*(E_{\Z})$.  
Furthermore, $r(\alpha)$ is always an infinite set.
\end{lem}

\begin{proof}
If $\alpha \in \CL^*(E_{\Z})$, then we can find $n, m \in \Z$ with $n < m$ and $\alpha = \omega_{[n, m]}$. 
Thus $\alpha^{-1} = \omega_{[-m-1, -n-1]} $ belongs to $\CL^*(E_{\Z})$. 
Moreover we can choose a large $l > 0$ so that $\alpha = \omega_{[n, m]}$ is a subpath of 
$$\omega_{[-2 \cdot 2^{2l}, 2 \cdot 2^{2l}-1]} = 1^{(2l)} 0^{(2l)} 0^{(2l)} 1^{(2l)} = 1^{(2l+1)}.$$
Then since 
$$\overline{\omega_{[-2 \cdot 2^{2l}, 2 \cdot 2^{2l}-1]}} = \overline{1^{(2l+1)}} = 0^{(2l+1)},$$
the labeled path $\overline{\alpha}$ is in $\CL^*(E_{\Z})$. 
Moreover $1^{(2l+1)}$ appears infinitely many times (see (\ref{defofmorse})), so that $r(\alpha)$ is infinite.
\end{proof}

\begin{remark}\label{overlap}
We emphasize from ~\cite{GH} that the (two-sided) Thue--Morse sequence is {\it overlap-free}: for any $\beta \in \CL^*(E_{\Z})$, the labeled path of the form $\beta \beta \beta_{[1,n]}$ with $n \leq |\beta|$ does not appear in $\CL^*(E_{\Z})$.
Thus $s_{\beta \beta \beta_{[1,n]}} = p_{r(\beta \beta \beta_{[1,n]})} = 0$.
For example, neither $1001000$ nor $1001001 = (100) (100) 1$ appear in the Thue--Morse sequence, and hence $100100 \notin \CL^*(E_{\Z})$.
\end{remark}

\begin{lem}\label{followerset}
If $\alpha$ has length $2^n$ and $i_1^{(n)} i_2^{(n)} \alpha \in \CL^*(E_{\Z})$ {\rm (}or $\alpha \, i_1^{(n)} i_2^{(n)} \in \CL^*(E_{\Z})${\rm )} for $i_1, i_2 = 0, 1$, then $\alpha=i^{(n)}$ for $i = 0$ or $1$. 
\end{lem}

\begin{proof}
We use an induction on $n$. 
When $n = 0$ or $1$, it is easy to see that the assertion is true.
Now we assume that the assertion holds for all $n \leq N$ ($N \geq 1$), and let 
$$\beta = i_1^{(N+1)} i_2^{(N+1)} \alpha \in \CL^*(E_{\Z})$$
with $|\alpha| = 2^{N+1}$. 
Then $\beta$ can be written as
$$\beta = i_1^{(N+1)} i_2^{(N+1)} \alpha = i_1^{(N)} \overline{i_1}^{(N)} i_2^{(N)} \overline{i_2}^{(N)} \alpha_1 \alpha_2$$
with $|\alpha_1| = |\alpha_2| = 2^N$.
From the induction hypothesis, $\alpha_1 \alpha_2$ must be of the form $j_1^{(N)} j_2^{(N)}$.
Thus it is enough to show that $\alpha_1 \alpha_2$ is neither $0^{(N)} 0^{(N)}$ nor $1^{(N)} 1^{(N)}$.

Suppose that $\alpha_1 \alpha_2 = 0^{(N)} 0^{(N)}$. 
Since $0^{(N)} 0^{(N)} 0^{(N)} \notin \CL^*(E_{\Z})$, the block $i_2^{(N)} \overline{i_2}^{(N)}$ must be $0^{(N)} 1^{(N)}$. 
If $i_1^{(N)} \overline{i_1}^{(N)} = 0^{(N)} 1^{(N)}$, then 
$$\beta = 0^{(N)} 1^{(N)} 0^{(N)} 1^{(N)} 0^{(N)} 0^{(N)}$$
and such $\beta$ can not be contained in $\CL^*(E_{\Z})$ (see Remark ~\ref{overlap}). 
If $i_1^{(N)} \overline{i_1}^{(N)} = 1^{(N)} 0^{(N)}$, we have
$$\beta = 1^{(N)} 0^{(N)} 0^{(N)} 1^{(N)} 0^{(N)} 0^{(N)},$$
but then $\beta \notin \CL^*(E_{\Z})$ again.
Therefore we obtain that $\alpha_1 \alpha_2 \neq 0^{(N)} 0^{(N)}$ and it follows that $\alpha_1 \alpha_2 \neq 1^{(N)} 1^{(N)}$ by considering $\overline{\beta}$. 
Hence $\alpha$ is either $0^{(N+1)}$ or $1^{(N+1)}$.

Finally, the case $\alpha \, i_1^{(n)} i_2^{(n)}$ can be done from $(\alpha \, i_1^{(n)} i_2^{(n)})^{-1} = j_1^{(n)} j_2^{(n)} \alpha^{-1} \in \CL^*(E_{\Z})$.
\end{proof}

Lemma ~\ref{followerset} implies that if we choose a long labeled path (long enough to contain at least two consecutive $i^{(n)}$-blocks), then its succeeding and preceding labeled paths must have certain forms. 
This is no longer true in general, for example $0^{(n)} \alpha \in \CL^*(E_{\Z})$ with $|\alpha| = 2^n$ does not imply $\alpha = 0^{(n)}$ or $1^{(n)}$ because $0^{(n-1)} 1^{(n-1)} 0^{(n-1)} 0^{(n-1)} = 0^{(n)} \alpha \in \CL^*(E_{\Z})$.  

One can observe from the definition of the Morse sequence that $\omega_{[2k, 2k+1]}$ is neither $00$ nor $11$.
Also any labeled path of length 5 must contain $00$ or $11$.
Thus if $\alpha = \alpha_1 \dots \alpha_{|\alpha|}$ is a path of length $|\alpha| \geq 5$, then we can actually figure out whether $\alpha_1 = \omega_{2k}$ or not (namely, $\alpha_1 = \omega_{2k+1}$) for some $k$.
Therefore if $\alpha = \alpha_1 \dots \alpha_5$, then $\alpha$ can be uniquely written in the form using $i^{(1)}$-blocks.  
For example, $\alpha = 01101$ then
$$
\alpha = (01) (10) 1 = 0^{(1)} 1^{(1)} 1
$$
whereas $\alpha = 0 (11) (01) \neq 0 \, i_1^{(1)} i_2^{(1)}$ for any $i_1, i_2 \in \{0, 1\}$.
We can extend this fact to the $i^{(n)}$-blocks for $n \geq 1$.

\begin{lem}\label{fiveblocks}
Let $\alpha \in \CL^*(E_{\Z})$ be a path of the form
$${i_1}^{(n)} {i_2}^{(n)} {i_3}^{(n)} {i_4}^{(n)} {i_5}^{(n)} $$
for some $n > 0$. 
Then $\alpha$ is written as one and only one of the following forms: 
\begin{equation}\label{rewriting}
{i_1}^{(n+1)} {i_3}^{(n+1)} {i_5}^{(n)} \hbox { \ or \ } {i_1}^{(n)} {i_2}^{(n+1)} {i_4}^{(n+1)}. 
\end{equation}
\end{lem}

\begin{proof}
The proof directly follows by substituting $i^{(n)}$ for $i$.
\end{proof}

By Lemma ~\ref{followerset} and Lemma ~\ref{fiveblocks}, we can write any labeled path in $\CL^*(E_{\Z})$ using $i^{(n)}$-blocks. 

\begin{prop}\label{prop-newform}
Let $\alpha$ be a labeled path with length $|\alpha| \geq 2$.
Then there exists $n \geq 0$ such that
\begin{equation}\label{newform}
\alpha = \gamma_0  i_1^{(n)} \dots i_k^{(n)} \gamma_1 , \quad 2 \leq k \leq 4,
\end{equation}
where $\gamma_0$ is a final subpath of $0^{(n)}$ or $1^{(n)}$, and $\gamma_1$ is an initial subpath of $0^{(n)}$ or $1^{(n)}$ with length $0 \leq |\gamma_0|, |\gamma_1| < 2^n$ {\rm(}$n$ is not necessarily unique{\rm)}.
Moreover, if we fix such an $n$, the expression of $\alpha$ in $i^{(n)}$-blocks is unique. 
\end{prop}

\begin{proof}
For $2 \leq |\alpha| \leq 4$, let $n=0$. 
If $|\alpha| = 5$, Lemma ~\ref{fiveblocks} shows that $n=1$ is the desired one. 

Now we assume that $6 \cdot 2^n \leq |\alpha| \leq 6 \cdot 2^{n+1}$ for $n \geq 0$.
Since $|i^{(n)}| = 2^n$ and $6 \leq |\alpha| / |i^{(n)}| \leq 12$, $\alpha$ contains at least $5$ and at most $12$ of $i^{(n)}$-blocks.
Thus $\alpha$ can be written as 
$$\alpha = \mu_0 i_1^{(n)} \dots i_k^{(n)} \mu_1 , \quad 5 \leq k \leq 12.$$
Applying Lemma ~\ref{followerset} and Lemma ~\ref{fiveblocks}, we can reduce the number of $i^{(n)}$-blocks to have
$$\alpha = \mu'_0 j_1^{(n+1)} \dots j_l^{(n+1)} \mu'_1 , \quad 2 \leq l \leq 6.$$
If $l \geq 5$, then we apply lemmas again and we see that $\alpha$ contains at least two and at most four blocks of $0^{(n)}$, $1^{(n)}$. 
Lemma ~\ref{followerset} implies that $\gamma_0$ and $\gamma_1$ are subpaths of $0^{(n)}$, $1^{(n)}$. 

We now prove the uniqueness of the expression. 
If $\alpha$ has two different expressions in $i^{(n)}$-blocks as in (\ref{newform}) for $n \geq 2$, then by Lemma ~\ref{fiveblocks}, $\alpha$ must have two different expressions in $i^{(n-1)}$-blocks. 
By induction, we conclude that $\alpha$ can be written as two different forms in $i^{(1)}$-blocks.
This contradicts to the fact that any labeled path $\alpha$ with $|\alpha| \geq 5$ must contain $00$ or $11$, and so it can be written in $i^{(1)}$-blocks is unique.
\end{proof}

\begin{ex}\label{ex-newform}

Let $\alpha = 001 0110 1001 0110 0110 1001 1001 0110 0$ be a labeled path in $\CL^*(E_{\Z})$. 
Note that $\alpha_{[1,2]} = 00$ is not equal to $0^{(1)}$ or $1^{(1)}$. 
Thus by Lemma ~\ref{fiveblocks}, $\alpha$ can be written in the following form using $i^{(1)}$-blocks 
\begin{align*} 
\alpha & \ =  001 0110 1001 0110 0110 1001 1001 0110 0\\ 
           & \ = 0 / 01 /01/10 /10/01/ 01/10/ 01/10 /10/01/ 10/01/ 01/10 /0 \\
          & \ = 0 / 0^{(1)} 0^{(1)} 1^{(1)} 1^{(1)} 0^{(1)} 0^{(1)} 1^{(1)} 0^{(1)} 1^{(1)} 1^{(1)} 0^{(1)} 1^{(1)} 0^{(1)} 0^{(1)} 1^{(1)} /0.
\end{align*}
(Note here that two consecutive $i_1^{(n)} i_2^{(n)}$ blocks between `/' form a $i_1^{(n+1)}$ block.)
Applying Lemma ~\ref{fiveblocks} repeatedly, we obtain 
\begin{align*} 
\alpha & \ = 0 / 0^{(1)} 0^{(1)} 1^{(1)} 1^{(1)} 0^{(1)} 0^{(1)} 1^{(1)} 0^{(1)} 1^{(1)} 1^{(1)} 0^{(1)} 1^{(1)} 0^{(1)} 0^{(1)} 1^{(1)} /0\\
            & \ = 0 0^{(1)} / 0^{(1)} 1^{(1)} / 1^{(1)} 0^{(1)} / 0^{(1)} 1^{(1)} / 0^{(1)} 1^{(1)} / 1^{(1)} 0^{(1)} / 1^{(1)} 0^{(1)} / 0^{(1)} 1^{(1)} /0 \\
           & \ = 0 0^{(1)} / 0^{(2)}  1^{(2)}   0^{(2)}   0^{(2)}  1^{(2)}  1^{(2)}  0^{(2)} /0\\
           & \ = 0 0^{(1)} 0^{(2)} / 1^{(2)}  0^{(2)}  / 0^{(2)} 1^{(2)}  / 1^{(2)}  0^{(2)} /0\\
           & \ = 0 0^{(1)} 0^{(2)} / 1^{(3)}   0^{(3)}  1^{(3)} /0\\
           & \ = \gamma_0  1^{(3)}   0^{(3)}  1^{(3)} \gamma_1,
\end{align*}           
where $\gamma_0 = 0 0^{(1)} 0^{(2)} = 0010010$ is a final subpath of $1^{(3)} = 10010110$, and $\gamma_1 = 0$ is an initial subpath of $0^{(3)} = 01101001$.
If $\alpha \beta_1 \in \CL^*(E_{\Z})$ for some labeled path $\beta_1$ of length $2^3 -1$, then by Lemma ~\ref{followerset}, $\gamma_1 \beta_1 = 0 \beta_1 = i^{(3)}$ for some $i = 0$ or $1$. 
Thus $0 \beta_1 = 0^{(3)}$. 
Similarly, if $\beta_0 \alpha \in \CL^*(E_{\Z})$ with $|\beta_0| = 1$ then $\beta_0 \gamma_0 = \beta_0 0 0^{(1)} 0^{(2)} = 1^{(3)}$, and we have $\beta_0 = 1$.
Thus
$$
\beta_0 \alpha \beta_1  \ =  1^{(3)} 1^{(3)} 0^{(3)} 1^{(3)} 0^{(3)}.
$$
Since $\omega_{[0,63]}  \ = 0^{(3)} 1^{(3)} 1^{(3)} 0^{(3)} 1^{(3)} 0^{(3)} 0^{(3)} 1^{(3)}$, we now see that $\beta_0 \alpha \beta_1 = \omega_{[8, 47]}$ and $\alpha = \omega_{[9,40]}$.
\end{ex}

\begin{notation}
We define $\CA^m \alpha \CA^n$ to be the set 
$$
\CA^m \alpha \CA^n := \{\beta_0 \alpha \beta_1 \in \CL^*(E_{\Z}) : |\beta_0| = m, |\beta_1| = n  \}
$$
for $m, n \geq 0$.
We identify $\CA^m \alpha$ with $\CA^m \alpha \CA^0$ and $\alpha \CA^n$ with $\CA^0 \alpha \CA^n$.
\end{notation}

We finish this section by proving the following lemma which is necessary in the next section to show that $C^*(E_{\Z}, \CL, \bE_{\Z})$ has a trace.

\begin{lem}\label{extension-number}
Let $\alpha$ be a labeled path with length $|\alpha| \geq 2$. Then the cardinality of the set $\CA \alpha \CA$ is $1, 2$, or $4$.
\end{lem}

\begin{proof}
We show that $|\CA \alpha \CA| \geq 3$ implies $|\CA \alpha \CA| = 4$. 
If $0 \alpha \notin \CL^*(E_{\Z})$, then $\CA \alpha \CA \subset \{1\alpha0, 1\alpha 1\}$ and $|\CA \alpha \CA| \leq 2$.
Similar arguments with the cases $1 \alpha \notin \CL^*(E_{\Z})$, $\alpha 0 \notin \CL^*(E_{\Z})$, or $\alpha 1 \notin \CL^*(E_{\Z})$ show that if $|\CA \alpha \CA| \geq 3$, then $\{ 0\alpha, 1\alpha, \alpha0, \alpha1 \} \subset \CL^*(E_{\Z})$. 

We know from Proposition ~\ref{prop-newform} that the path $\alpha$ can be written as
$$\alpha = \gamma_0 i_1^{(n)} \dots i_k^{(n)} \gamma_1 , \quad 2 \leq k \leq 4$$ 
for some $n \geq 0$.
If $0 < |\gamma_1| < 2^n$, Proposition ~\ref{prop-newform} also shows that $\gamma_1 j_1$ ($j_1 = 0$ or $1$) is an initial subpath of $i^{(n)}$ for some $i = 0$ or $1$, and thus $j_1$ is uniquely determined by $\gamma_1$.
This implies that $|\CA \alpha \CA| \leq 2$. 
Hence when $|\CA \alpha \CA| \geq 3$, $\gamma_1$ should be an empty word and similarly, so is $\gamma_0$. 
Thus if $|\CA \alpha \CA| \geq 3$, we can write  
$$\alpha = {i_1}^{(n)} \dots {i_k}^{(n)}, \quad 2 \leq k \leq 4,$$
for some $n$.
It is an easy but time-consuming work to show that there exist only four labeled paths $\alpha$ 
$$0^{(n)} 1^{(n)}, \, 1^{(n)} 0^{(n)}, \,  0^{(n)} 1^{(n)} 1^{(n)} 0^{(n)}, \,  1^{(n)} 0^{(n)} 0^{(n)} 1^{(n)}$$
satisfying $|\CA \alpha \CA| \geq 3$, and in each case one can show that $|\CA \alpha \CA| = 4$ by a straightforward computation. 
\end{proof}

\section{The trace on the labeled graph $C^*$-algebra $C^*(E_{\Z}, \CL, \bE_{\Z})$}\label{4} 

\noindent Note that if $\tau : C^*(E_{\Z}, \CL, \bE_{\Z}) \to \C$ is a tracial state, it satisfies
$$
\tau(s_{\alpha} p_{r(\beta \alpha)} s_{\alpha}^*) = \tau(p_{r(\beta \alpha)}),
$$
$$
\tau(p_{r(\alpha)}) = \tau(s_0 p_{r(\alpha 0)} s_0^*) + \tau( s_1 p_{r(\alpha 1)} s_1^*) = \tau(p_{r(\alpha 0)}) + \tau(p_{r(\alpha 1)})
$$
for $p_{r(\alpha)}, s_{\alpha} p_{r(\beta \alpha)} s_{\alpha}^* \in C^*(E_{\Z}, \CL, \bE_{\Z})^{\gamma}$. 
Let  
$$
\chi := \overline{{\rm span}} \{p_A : A \in \bE_{\Z} \}
$$
be the commutative $C^*$-subalgebra of $C^*(E_{\Z}, \CL, \bE_{\Z})^{\gamma}$ generated by projections $p_A$, $A \in \bE_{\Z}$. 
The following lemma shows that every faithful state on $C^*(E_{\Z}, \CL, \bE_{\Z})^{\gamma}$ (or on $\chi$) with the above  property extends to a tracial state on $C^*(E_{\Z}, \CL, \bE_{\Z})$.
By $\mathcal{S}(\mathcal{D})$ we denote the set of all faithful tracial states on a $C^*$-algebra $\mathcal{D}$. 

\begin{lem}\label{S_1, S_2, S_3}
Let $\mathcal{S} := \mathcal{S}(C^*(E_{\Z}, \CL, \bE_{\Z}))$ be the set of tracial states of $C^*(E_{\Z}, \CL, \bE_{\Z})$ and let
\begin{align*} 
\mathcal{S}_{\gamma} & := \{ \tau \in \mathcal{S}(C^*(E_{\Z}, \CL, \bE_{\Z})^{\gamma}) : \tau(s_\alpha p_{r(\beta \alpha)} s_\alpha^*) = \tau(p_{r(\beta \alpha)}) \},\\
\mathcal{S}_{\chi} & := \{ \tau \in \mathcal{S}(\chi) : \tau(p_{r(\alpha)}) = \tau(p_{r(\alpha 0)}) + \tau(p_{r(\alpha 1)})\}.
\end{align*} 
Then the restriction maps $\phi_1 : \mathcal{S} \to \mathcal{S}_{\gamma}$ and $\phi_2 : \mathcal{S}_{\gamma} \to \mathcal{S}_{\chi}$ are bijections with the inverses $\phi_1^{-1}(\tau_1) = \tau_1 \circ \Psi$ for $\tau_1 \in \mathcal{S}_{\gamma}$ and $\phi_2^{-1}(\tau_2) (s_\alpha p_{r(\beta \alpha)} s_\alpha^*) = \tau_2 (p_{r(\beta \alpha)})$ for $\tau_2 \in \mathcal{S}_{\chi}$, respectively.
\end{lem}

\begin{proof}
To prove that $\phi_1$ is injective, it is enough to show that 
$$ \tau = \tau \circ \Psi \mbox{ for all } \tau \in \mathcal{S}.$$
Let $s_{\alpha} p_A s_{\beta}^* \in C^*(E_{\Z}, \CL, \bE_{\Z})$ with $|\alpha| \neq |\beta|$. 
Without loss of generality, we may assume that $|\alpha| > |\beta|$. Since $\tau$ is a trace, we have
$$\tau (s_{\alpha} p_A s_{\beta}^*) = \tau (s_{\beta}^* s_{\alpha} p_A) 
= \left\{ \begin{array}{ll} \tau (p_{r(\beta)} s_{\alpha_1} p_A) = \tau (s_{\alpha_1} p_A) & \hbox{ if } \alpha = \beta \alpha_1, \\
                                                                   0 & \hbox{ otherwise,}
                                         \end{array}
                                \right.$$
and then
$$
\tau (s_{\alpha_1} p_A) = \tau (p_A s_{\alpha_1}) = \tau (s_{\alpha_1} p_{r(A, \alpha_1)}) = \cdots = \tau (s_{\alpha_1} p_{r(A, \alpha_1^3)}) = 0,
$$ 
because $\alpha_1^3 \notin \CL^*(E_{\Z})$. Thus $\tau \equiv 0$ on $C^*(E_{\Z}, \CL, \bE_{\Z}) \setminus C^*(E_{\Z}, \CL, \bE_{\Z})^{\gamma}$ and hence we have $ \tau = \tau \circ \Psi$.

If we prove that $\tau \circ \Psi \in \mathcal{S}$ for all $\tau \in \mathcal{S}_{\gamma}$, then $\phi_1$ is surjective.
By the continuity of $\tau$ and $\Psi$, it is enough to show that 
\begin{equation}\label{XY}
\tau(\Psi(XY)) = \tau(\Psi(YX))
\end{equation}
for $X, Y \in {\rm span}\{s_{\alpha} p_A s_{\beta}^* : \alpha, \beta \in \CL^*(E), A \subset r(\alpha) \cap r(\beta)\}$.
Since $\tau$ and $\Psi$ is linear, we are reduced to proving (\ref{XY}) for $X = s_{\alpha} p_A s_{\beta}^*$, $Y = s_{\mu} p_B s_{\nu}^*$.
Moreover, the equality $s_{\alpha} p_A s_{\beta}^* = \sum_{\delta \in \CL(A E^n)} s_{\alpha \delta} p_{r(A, \delta)} s_{\beta \delta}^*$ yields that we only need to prove (\ref{XY}) when $|\beta| = |\mu|$.  
Note from 
$$
\gamma_z (XY) = z^{|\alpha|-|\beta|+|\mu|-|\nu|} XY \mbox{ and } \gamma_z (YX) = z^{-|\alpha|+|\beta|-|\mu|+|\nu|} YX
$$ that  
$$
|\alpha|-|\beta|+|\mu|-|\nu| = |\alpha|-|\nu| \neq 0 \Longrightarrow \Psi(XY) = \Psi(YX) = 0.
$$
Hence we may assume that $|\alpha| = |\nu|$ and $|\beta| = |\delta|$.
Then we obtain that
\begin{align*}
\tau(\Psi(XY)) & = \tau (XY) = \tau (s_{\alpha} p_A s_{\beta}^* s_{\mu} p_B s_{\nu}^* ) \\
&= \tau (\delta_{\beta,\mu} s_{\alpha} p_{A \cap B} s_{\nu}^* ) = \tau (\delta_{\beta,\mu} \delta_{\alpha, \nu} s_{\alpha} p_{A \cap B} s_{\alpha}^*)\\
& = \tau (\delta_{\beta,\mu} \delta_{\alpha, \nu} p_{A \cap B}),\\
\tau(\Psi(YX)) & = \tau (\delta_{\beta,\mu} \delta_{\alpha, \nu} p_{B \cap A})\\
&= \tau(\Psi(XY)),
\end{align*}
and $\phi_1$ is bijective.

Next, we prove that $\phi_2$ is bijective. 
Let $\tau \in \mathcal{S}_{\chi}$.
Recall from Remark ~\ref{remark_AF} (i) that $C^*(E_{\Z}, \CL, \bE_{\Z})^{\gamma} = \displaystyle \lim_{\longrightarrow} (F_k, \iota_k)$ where
$$
F_k = \overline{\rm{span}}\{ s_{\alpha} p_{r(\beta \alpha)} s_{\alpha}^* : \alpha, \beta \in \CL^*(E_{\Z}) \mbox{ with } |\alpha|=|\beta| = k \}  \cong \mathbb{C}^{d_k}
$$
with $d_k = |\CL^*(E_{\Z}^{2k})|$. 
Define $\overline{\tau}_k : F_k \to \C$ by $\overline{\tau}_k (s_\alpha p_{r(\beta \alpha)} s_\alpha^*) := \tau (p_{r(\beta \alpha)})$.
Then $\overline{\tau}_k$ is a state on $F_k$ for each $k \geq 1$.
Also for the inclusion $\iota_k : F_k \to F_{k+1}$, we have $\overline{\tau}_k (x) = \overline{\tau}_{k+1} (\iota_k (x))$ for $x \in F_k$ and $k \geq 1$.
In fact, 
\begin{align*} 
\overline{\tau}_{k+1} (\iota_k (s_{\alpha} p_{r(\beta \alpha)} s_{\alpha}^*)) 
& =\overline{\tau}_{k+1} (\sum_{a \in \{0,1\}} s_{\alpha a} p_{r(\beta \alpha a)} s_{\alpha a}^*) \\
& =\sum_{a \in \{0,1\}} \tau (p_{r(\beta \alpha a)}) = \tau (p_{r(\beta \alpha)}) \\
& = \overline{\tau}_k (s_{\alpha} p_{r(\beta \alpha)} s_{\alpha}^*).
\end{align*}           
Thus $\overline{\tau} := \displaystyle \lim_{\longrightarrow} \overline{\tau}_k$ extends to a state on $C^*(E_{\Z}, \CL, \bE_{\Z})^{\gamma}$.
Moreover $\overline{\tau} (s_{\alpha} p_{r(\beta \alpha)} s_{\alpha}^*) = \tau (p_{r(\beta \alpha)}) = \overline{\tau}  (p_{r(\beta \alpha)})$, and hence $\overline{\tau} \in \mathcal{S}_{\gamma}$. 
It is now obvious that $\phi_2$ is bijective.
\end{proof}

\begin{remark}
Let $\tau \in \mathcal{S}_{\chi}$.
The equality $r(\alpha) = \bigsqcup_{\mu \alpha \in \CA^m \alpha} r(\mu \alpha)$ implies that 
$$
\tau(p_{r(\alpha)}) = \sum_{\mu \alpha \in \CA^m \alpha} \tau(p_{r(\mu \alpha)}).
$$
In addition, $p_{r(\alpha)} = \sum_{\alpha \mu \in \alpha \CA^n} s_{\mu} p_{r(\alpha \mu)} s_{\mu}^*$ implies that 
$$\tau(p_{r(\alpha)}) = \sum_{\alpha \mu \in \alpha \CA^n} \tau(p_{r(\alpha \mu)}).$$ 
Combining these two, we obtain 
$$
\tau(p_{r(\alpha)}) = \sum_{\mu \in \CA^m \alpha \CA^n} \tau(p_{r(\mu)})
$$
for any $m,n \geq 0$.
For example, let $\alpha = 0101$. 
Since $\CA^2 \alpha = \{100101\}$ and $\CA^2 \alpha \CA^2 = \{10010110\}$, we have
$$
\tau(p_{r(0101)}) = \tau(p_{r(100101)}) = \tau(p_{r(10010110)}).
$$
\end{remark}

By (\ref{accommodating set form}), any linear functional on $\chi = \overline{{\rm span}} \{P_A : A \in \bE_{\Z} \}$ is determined by its values on the projections $p_{r(\alpha)}$, $\alpha \in \CL^*(E_{\Z})$.
Note that $\chi$ is a commutative $C^*$-algebra which is the inductive limit of the finite dimensional $C^*$-subalgebras $\chi_k := {\rm span} \{P_{r(\alpha)} : |\alpha| = k \} \cong \C^{|\CL(E^k)|}$.
Thus any continuous linear functional $\psi : \cup_{k \geq 1} \chi_k \to \C$ naturally extends to a linear functional on $\chi$.  
Let $\phi$ be a linear functional on $\cup_{k \geq 3} \chi_k$ given by 
\begin{enumerate}
\item[(i)]  $\phi(p_{r(\alpha)}) = 1 / 6$ for $\alpha$ with $|\alpha| = 3$,
\item[(ii)] for $\alpha$ with $|\alpha| \geq 3$,
$$\phi (p_{ r(i \alpha)} ) = \left\{ \begin{array}{ll} \phi (p_{r(\alpha)}), & \hbox{ if } \quad \overline{i} \alpha \notin \CL^*(E_{\Z}), \\

                                                                    \phi (p_{r(\alpha)}) / 2, & \hbox{ if } \quad \overline{i} \alpha \in \CL^*(E_{\Z}).
                                         \end{array}
                                \right.$$
\end{enumerate}
Then $\phi$ is positive and of norm $1$, hence it extends to a state on $\chi$. 
Note that (ii) holds for all paths $\alpha \in \CL^*(E_{\Z})$ except $\alpha = 0, 1$. 
The paths $0, 1$ are the only paths $\alpha$ such that $|\CA \alpha \CA| = 3$. 

\begin{thm}\label{main}
The state $\phi$ on $\chi$ extends to a trace $\tilde{\phi}$ on $C^*(E_{\Z}, \CL, \bE_{\Z})$ by
$$
\tilde{\phi} (s_{\alpha} p_{A} s_{\beta}^*) = \delta_{\alpha, \beta} \, \phi(p_{A})
$$ 
for $A \subset r(\alpha) \cap r(\beta)$. 
In particular, 
\begin{equation*}
\tilde{\phi}(p_{r(\alpha)}) = \begin{cases} \frac{1}{3},  \quad \mbox{ if } \alpha = 00, 11,\\
\frac{1}{6},  \quad \mbox{ if } \alpha = 01, 10,\\
\frac{1}{6},  \quad \mbox{ if } |\alpha| = 3,
\end{cases}
\end{equation*}
and $\tilde{\phi}(p_{r(\alpha)}) = 2 \tilde{\phi}(p_{r(0 \alpha)}) = 2 \tilde{\phi}(p_{r(1 \alpha)})$ for $\alpha \in \CL(E^{\geq 3})$ with $0\alpha, 1\alpha \in \CL(E^{\geq 4})$.
\end{thm}

\begin{proof}
We will prove that the state $\phi$ on $\chi$ belongs to $\mathcal{S}_{\chi}$. 
Then the assertion follows from Lemma ~\ref{S_1, S_2, S_3}.
For this we first show that for any labeled path $\alpha$, 
$$
\phi(p_{r(\alpha)}) = \phi (p_{r(\alpha^{-1})}).
$$
We use the induction on the length of $\alpha$. 
From the definition of $\phi$, it is obvious that $\phi(p_{r(\alpha)}) = \phi (p_{r(\alpha^{-1})})$ for $|\alpha| \leq 3$. 
We now assume that this holds for all labeled path $\alpha$ with length $|\alpha| \leq N$ for some $N \geq 3$. 
Let $j_0 \beta j_1 \in \CL^*(E_{\Z})$ for some $j_0, j_1 = 0$ or $1$, and $|\beta| = N-1$.
By Lemma ~\ref{extension-number}, we know that $\CA \beta \CA$ is one of the following: 
\begin{enumerate}
\item[(i)] $\{j_0 \beta j_1\}$,
\item[(ii)] $\{j_0 \beta j_1, \, \overline{j_0} \beta j_1\}$, 
\item[(iii)] $\{j_0 \beta j_1, \, j_0 \beta \overline{j_1}\}$,
\item[(iv)] $\{j_0 \beta j_1, \, \overline{j_0} \beta \overline{j_1}\}$, and
\item[(v)] $\{j_0 \beta j_1,\, \overline{j_0} \beta j_1,\, j_0 \beta \overline{j_1},\, \overline{j_0} \beta \overline{j_1})\}$.
\end{enumerate}
Also for each case, we can show that $\phi (p_{r(j_0 \beta j_1)}) = \phi (p_{r(j_1 \beta^{-1} j_0)})$.
Actually, if $\CA \beta \CA = \{j_0 \beta j_1\}$ as in (i), first note that $\overline{j_0} \beta , \,  \beta \overline{j_1} \notin \CL^*(E_{\Z})$ and so $\beta^{-1} \overline{j_0} , \,\, \overline{j_1} \beta^{-1} \notin \CL^*(E_{\Z})$.
Thus 
\begin{align*}
\phi (p_{r(j_0 \beta j_1)}) & = \phi (p_{r(\beta j_1)}) = \phi (p_{ r(j_1 \beta^{-1})})  \\
                                     & = \phi (p_{r(\beta^{-1})}) = \phi (p_{r(\beta)}) \\
                                     & = \phi (p_{r(j_0 \beta)}) = \phi (p_{r(\beta^{-1} j_0)}) \\
                                     & = \phi (p_{r(j_1 \beta^{-1} j_0)}).
\end{align*}
In case (ii), we have
\begin{align*}
\phi (p_{r(j_0 \beta j_1)}) & = \phi (p_{r(\beta j_1)}) / 2 = \phi (p_{r(j_1 \beta^{-1})}) /2 \\
                                     & = \phi (p_{r(\beta^{-1})}) /2 = \phi (p_{r(\beta)}) /2\\
                                     & = \phi (p_{r(j_0 \beta)}) = \phi (p_{r(\beta^{-1} j_0)}) \\
                                     & = \phi (p_{r(j_1 \beta^{-1} j_0)}).
\end{align*}
The rest cases can be shown similarly. Then from the above observation, we have
$$
\phi (p_{r(\alpha)}) = \phi (p_{r(\alpha^{-1})}) = \phi (p_{r(0 \alpha^{-1})})  + \phi (p_{r(1 \alpha^{-1})}) = \phi (p_{r(\alpha 0)})  + \phi (p_{r(\alpha 1)}),
$$
which shows that $\phi \in \mathcal{S}_{\chi}$. 
This completes the proof.
\end{proof}

\begin{prop}\label{unique trace}
The trace $\tilde{\phi}$ on the $C^*$-algebra $C^*(E_{\Z}, \CL, \bE_{\Z})$ is a unique trace.
\end{prop}

\begin{proof}
By Lemma ~\ref{S_1, S_2, S_3} and Theorem ~\ref{main}, it is enough to show that $|\mathcal{S}_{\chi}| \leq 1$. 
Let $\tau \in \mathcal{S}_{\chi}$ and set
\begin{align*}
b_{n,1} := \tau(p_{r(0^{(n)} 0^{(n)})}), & \quad b_{n,2} := \tau(p_{r(0^{(n)} 1^{(n)})}),\\
b_{n,3} := \tau(p_{r(1^{(n)} 0^{(n)})}), & \quad b_{n,4} := \tau(p_{r(1^{(n)} 1^{(n)})})
\end{align*}
for $n \geq 0$.
(If two states on $\chi$ have the common values at $p_{r(i^{(n)} j^{(n)})}$ for all $i, j = 0, 1$ and $n \geq 0$, then they are the same state.)
We first show $b_{n,1} = b_{n,4} < b_{n,2} = b_{n,3}$.
This particularly gives $b_{0,1} + b_{0,2} = \frac{1}{2}$.
By Lemma ~\ref{followerset}, we have 
\begin{align}\label{bn}
b_{n,2} = \tau (p_{r(0^{(n)} 1^{(n)})}) & = \tau (p_{r(0^{(n)} 1^{(n)} 0^{(n)})}) + \tau (p_{r(0^{(n)} 1^{(n)} 1^{(n)})}) \nonumber \\
                                     & = \tau (p_{r(0^{(n)} 1^{(n)} 0^{(n)})}) + \tau (p_{r(1^{(n)} 1^{(n)})}) \nonumber\\
                                     & = \tau (p_{r(0^{(n)} 1^{(n)} 0^{(n)})}) + b_{n,4},\\
b_{n,3} = \tau (p_{r(1^{(n)} 0^{(n)})}) & = \tau (p_{r(0^{(n)} 1^{(n)} 0^{(n)})}) + \tau (p_{r(1^{(n)} 1^{(n)} 0^{(n)})}) \nonumber\\
                                     & = \tau (p_{r(0^{(n)} 1^{(n)} 0^{(n)})}) + \tau (p_{r(1^{(n)} 1^{(n)})}) \nonumber\\
                                     & = \tau (p_{r(0^{(n)} 1^{(n)} 0^{(n)})}) + b_{n,4},\nonumber
\end{align}
and thus we have $b_{n,2} - b_{n,4} = b_{n,3} - b_{n,4} = \tau (p_{r(0^{(n)} 1^{(n)} 0^{(n)})}) > 0$, which shows $b_{n,2}=b_{n,3} > b_{n,4}$. 
Similarly $b_{n,2} > b_{n,1}$ can be shown, and from 
\begin{align*}
b_{n,1} = \tau (p_{r(0^{(n)} 0^{(n)})}) & = \tau (p_{r(1^{(n)} 0^{(n)} 0^{(n)} 1^{(n)})}) =  \tau (p_{r(1^{(n+1)} 0^{(n+1)})}) = b_{n+1, 3} \, , \\
b_{n,4} = \tau (p_{r(1^{(n)} 1^{(n)})}) & = \tau (p_{r(0^{(n)} 1^{(n)} 1^{(n)} 0^{(n)})}) =  \tau (p_{r(0^{(n+1)} 1^{(n+1)})}) = b_{n+1, 2} \, ,
\end{align*}
we also see that $b_{n,1} = b_{n+1,3} = b_{n+1,2} = b_{n,4}$.
Thus the numbers $b_{n,k}$ for $n \geq 0$, $k = 1, 2, 3, 4$ are determined in a unique manner once $b_{0,1}$ is given.

Now (\ref{bn}) gives us 
\begin{align*}
b_{n,2} & = \tau (p_{r(0^{(n)} 1^{(n)} 0^{(n)})}) + b_{n,4} \\
            & = \tau (p_{r(0^{(n)} 0^{(n)} 1^{(n)} 0^{(n)})}) + \tau (p_{r(1^{(n)} 0^{(n)} 1^{(n)} 0^{(n)})}) + b_{n,1} \\
            & = \tau (p_{r(1^{(n)} 0^{(n)} 0^{(n)} 1^{(n)} 0^{(n)})}) + \tau (p_{r(1^{(n+1)} 1^{(n+1)})}) + b_{n,1} \\
            & = \tau (p_{r(1^{(n+1)} 0^{(n+1)} 0^{(n+1)})}) + b_{n+1,4} + b_{n,1} \\
            & = 2 b_{n+1,1} + b_{n,1},
\end{align*}
and so
$$
\left( \begin{array}{cc}
b_{n,1} \\
b_{n,2} \\
\end{array} \right)
 =
\left( \begin{array}{cc}
0 & 1 \\
2 & 1 \\
\end{array} \right)
\left( \begin{array}{cc}
b_{n+1,1} \\
b_{n+1,2} \\
\end{array} \right).
$$
Since $b_{0,1} + b_{0,2} = \frac{1}{2}$, if we put $t = b_{0,1}$, 
\begin{align}\label{tracevalue} 
\left( \begin{array}{cc}
b_{n,1} \\
b_{n,2} \\
\end{array} \right)
& =
\left( \begin{array}{cc}
-1/2 & 1/2 \\
1 & 0 \\
\end{array} \right)^n
\left( \begin{array}{cc}
b_{0,1} \\
b_{0,2} \\
\end{array} \right) \nonumber \\
& =
\left(
\left( \begin{array}{cc}
1/\sqrt{3} & 1/\sqrt{3} \\
2/\sqrt{3} & -1/\sqrt{3} \\
\end{array} \right)
\left( \begin{array}{cc}
1/2 & 0 \\
0 & -1 \\
\end{array} \right)
\left( \begin{array}{cc}
1/\sqrt{3} & 1/\sqrt{3} \\
2/\sqrt{3} & -1/\sqrt{3} \\
\end{array} \right)
\right)^n
\left( \begin{array}{cc}
b_{0,1} \\
b_{0,2} \\
\end{array} \right) \nonumber \\
& =
\frac{1}{3}
\left( \begin{array}{cc}
1 & 1 \\
2 & -1 \\
\end{array} \right)
\left( \begin{array}{cc}
1/2^n & 0 \\
0 & (-1)^n \\
\end{array} \right)
\left( \begin{array}{cc}
1 & 1 \\
2 & -1 \\
\end{array} \right)
\left( \begin{array}{cc}
t \\
1/2 - t \\
\end{array} \right) \nonumber\\
& =
\frac{1}{3}
\left( \begin{array}{cc}
(1/2)^{n+1} + (3 t - 1/2)(-1)^n \\
(1/2)^{n} - (3 t - 1/2)(-1)^n \\
\end{array} \right).
\end{align}
The state $\tau$ is faithful, so $b_{n,1}$ and $b_{n,2}$ must be strictly positive.
Thus $t = 1/6$, and hence $\tau = \phi$.  
\end{proof}

\begin{remark}\label{traceequation}
By (\ref{tracevalue}), the unique trace $\tilde{\phi}$ on $C^*(E_{\Z}, \CL, \bE_{\Z})$ satisfies
\begin{align}\label{trace_equation}
\tilde{\phi} (p_{r(0^{(n)} 0^{(n)})}) & = \tilde{\phi} (p_{r(1^{(n)} 1^{(n)})}) = 1/ (6 \cdot 2^{n}) \hbox{, and } \nonumber\\
\tilde{\phi} (p_{r(0^{(n)} 1^{(n)})}) & = \tilde{\phi} (p_{r(1^{(n)} 0^{(n)})}) = 1/ (3 \cdot 2^{n}).
\end{align}
By Proposition ~\ref{prop-newform}, every labeled path $\alpha$ with $|\alpha| \geq 2$ can be written in the form $\alpha = \gamma_0  i_1^{(n)} \dots i_k^{(n)} \gamma_1$ for $2 \leq k \leq 4$. 
We thus can compute $\tilde{\phi} (p_{r(\alpha)})$ for all labeled paths $\alpha$ using (\ref{trace_equation}).
For example, $\alpha = x_{[10, 31]}$ is written as
$$
\alpha = 0 1 0^{(2)} 1^{(2)} 0^{(2)} 0^{(2)} 1^{(2)} = 0 1 0^{(2)} 1^{(3)} 0^{(3)},
$$
and then $\tilde{\phi}(p_{r(\alpha)}) = \tilde{\phi}(p_{r(1^{(3)} 1^{(3)} 0^{(3)})}) = \tilde{\phi}(p_{r(1^{(3)} 1^{(3)})}) = 1/ (6 \cdot 2^3)$.
\end{remark}

\section{$K$-groups of $C^*(E_{\Z}, \CL, \bE_{\Z})$}\label{5}

\noindent Note that for $i^{(n)}$-blocks with a even number $n$, the first and last alphabets of $i^{(n)}$ are exactly $i$.
For example, $0^{(2)} = \dot{0}11\dot{0}$ but $0^{(3)} = \dot{0}110100\dot{1}$.    
Therefore it is convenient to present the preceding alphabet of labeled paths:
$$\CA^{2^2} 1^{(2)} 1^{(2)} =\{0^{(2)} 1^{(2)} 1^{(2)} \} \mbox{ and } \CA \, 1^{(2)} 1^{(2)} =\{0 \, 1^{(2)} 1^{(2)} \},$$
while $\CA^{2^3} 1^{(3)} 1^{(3)} =\{0^{(3)} 1^{(3)} 1^{(3)} \}$ and $\CA \, 1^{(3)} 1^{(3)} =\{1 \, 1^{(3)} 1^{(3)} \}$.

We first briefly review the K-groups of labeled graph $C^*$-algebras.
We let $\Omega_l = \{[v]_l :  v \in E^0 \}$, and $\Omega = \cup_{l \geq 1} \Omega_l$ be the set of all generalized vertices. 
By $\Z(\Omega)$ we denote the additive group ${\rm span}_{\Z} \{\chi_{[v]_l} : [v]_l \in \Omega \}$. 

Let $({\rm I} - \Phi) : \Z(\Omega) \to \Z(\Omega)$ be the linear map defined by
$$
({\rm I} - \Phi) (\chi_{[v]_l}) = \chi_{[v]_l} - \sum_{a \in \CA} \chi_{r([v]_l, a)} \quad  \mbox{ for } \quad [v]_l \in \Omega.
$$
Then the following is known in ~\cite[Corollary 8.3]{BCP}.
\begin{enumerate}
\item[(i)]  $K_1(C^*(E, \CL, \bE)) \cong \ker({\rm I}-\Phi)$, 
\item[(ii)] $K_0(C^*(E, \CL, \bE)) \cong {\rm coker}({\rm I}-\Phi)$ via the map $[p_{[v]_l}]_0 \mapsto \chi_{[v]_l} + {\rm Im}({\rm I}-\Phi)$.
\end{enumerate}

As shown in ~\cite{JKKP}, $C^*(E_{\Z}, \CL, \bE_{\Z})$ is isomorphic to the crossed product $C(X) \rtimes_\sigma \Z$ of a two-sided Thue--Morse subshift $(X, \sigma)$, and hence 
$$
K_1(C^*(E_{\Z}, \CL, \bE_{\Z})) \cong \Z
$$ 
follows from ~\cite{GPS}.
To compute $K_0$-group of $C^*(E_{\Z}, \CL, \bE_{\Z})$, first note that 
$$
\sum_{a \in \CA} p_{r(a \alpha)} = p_{r(\alpha)} = \sum_{a \in \CA}s_a p_{r(\alpha a)} s_a^*
$$ 
yields that
\begin{equation}\label{R1}
\sum_{a \in \CA} [p_{r(a \alpha)}]_0 = [p_{r(\alpha)}]_0 = \sum_{a \in \CA} [p_{r(\alpha a)}]_0. 
\end{equation}
Combining (\ref{R1}) and Proposition ~\ref{prop-newform}, we can rewrite $[p_{r(\alpha)}]_0$ as
$$
\sum k_i [p_{r(\beta_i)}]_0
$$
where $k_i \in \N$ and $\beta_i \in \{0^{(n)}0^{(n)}1^{(n)}, 0^{(n)}1^{(n)}0^{(n)}, \dots, 1^{(n)}1^{(n)}0^{(n)} \}$ for some $n \geq 0$.
(This expression need not be unique.)
We begin with a relation between $[p_{\beta_i}]$'s.

\begin{lem}\label{lem a_n b_n}
For all $n \geq 0$, we have
\begin{align}
[p_{r(0^{(n)}1^{(n)}0^{(n)})}]_0 &= [p_{r(1^{(n)}0^{(n)}1^{(n)})}]_0,\label{eqn a_n} \\
[p_{r(0^{(n)}0^{(n)}1^{(n)})}]_0 = [p_{r(1^{(n)}0^{(n)}0^{(n)})}]_0 &= [p_{r(0^{(n)}1^{(n)}1^{(n)})}]_0 = [p_{r(1^{(n)}1^{(n)}0^{(n)})}]_0.\label{eqn b_n}
\end{align}
Moreover, if we let
\begin{align}\label{def a_n b_n}
a_n := [p_{r(0^{(n)}1^{(n)}0^{(n)})}]_0 \quad \mbox{and} \quad b_n := [p_{r(0^{(n)}0^{(n)}1^{(n)})}]_0, 
\end{align}
then we have $a_n = 2 b_{n+1}, \, b_n = a_{n+1} + b_{n+1}$, and $a_n \neq b_n$.
\end{lem}

\begin{proof}
First notice that the equality (\ref{R1}) yields that 
$$
[p_{r(1^{(n)}0^{(n)}0^{(n)})}]_0 = [p_{r(0^{(n)}0^{(n)})}]_0 = [p_{r(0^{(n)}0^{(n)}1^{(n)}}]_0
$$
and
$$
[p_{r(0^{(n)}1^{(n)}1^{(n)})}]_0 = [p_{r(1^{(n)}1^{(n)})}]_0 = [p_{r(1^{(n)}1^{(n)}0^{(n)}}]_0.
$$
The equality
\begin{align}\label{b_n}
[p_{r(0^{(n)}0^{(n)})}]_0 & = [p_{r(1^{(n)}0^{(n)}0^{(n)}1^{(n)})}]_0 = [p_{r(1^{(n+1)}0^{(n+1)})}]_0 \nonumber \\
& = [p_{r(1^{(n+1)}0^{(n+1)}0^{(n+1)})}]_0 + [p_{r(1^{(n+1)}0^{(n+1)}1^{(n+1)})}]_0 \\
& = [p_{r(0^{(n+1)}0^{(n+1)}1^{(n+1)})}]_0 + [p_{r(1^{(n+1)}0^{(n+1)}1^{(n+1)})}]_0 \nonumber \\
& = [p_{r(0^{(n+1)}1^{(n+1)}}]_0 =  [p_{r(1^{(n)}1^{(n)})}]_0 \nonumber
\end{align}
shows that (\ref{eqn b_n}) holds.
From this, we obtain that
\begin{align}\label{a_n}
[p_{r(0^{(n)}1^{(n)}0^{(n)})}]_0 & = [p_{r(0^{(n)}1^{(n)}0^{(n)}0^{(n)})}]_0 +  [p_{r(0^{(n)}1^{(n)}0^{(n)}1^{(n)})}]_0 \nonumber\\
& = [p_{r(1^{(n)}0^{(n)}1^{(n)}0^{(n)}0^{(n)}1^{(n)})}]_0 + [p_{r(0^{(n+1)}0^{(n+1)})}]_0 \\
& = [p_{r(1^{(n+1)}1^{(n+1)}0^{(n+1)})}]_0 + [p_{r(0^{(n+1)}0^{(n+1)})}]_0 \nonumber \\
& = 2 [p_{r(0^{(n+1)}0^{(n+1)})}]_0, \nonumber
\end{align}
and similarly
$$
[p_{r(1^{(n)}0^{(n)}1^{(n)})}]_0 = 2 [p_{r(0^{(n+1)}0^{(n+1)})}]_0,
$$
so (\ref{eqn a_n}) holds.

We already obtain from (\ref{b_n}) that $b_n = a_{n+1} + b_{n+1}$ and from (\ref{a_n}) that $a_n = 2 b_{n+1}$.
Combining these we get 
$$
a_n - b_n = - (a_{n+1} - b_{n+1})
$$
which then implies that $a_n - b_n$ does not belong to ${\rm Im}({\rm I}-\Phi)$, and hence $a_n \neq b_n$.
\end{proof}

As a result, we have the following. 

\begin{prop}
Let $a_n$, $b_n$ ($n \geq 0$) be elements in $K_0(C^*(E_{\Z}, \CL, \bE_{\Z}))$ given by (\ref{def a_n b_n}).
Then the $K_0$-group of $C^*(E_{\Z}, \CL, \bE_{\Z})$ is
\begin{align*}
K_0(C^*(E_{\Z}, \CL, \bE_{\Z})) & = {\rm span}_{\Z}\{a_n, b_n : n \geq 0, \, a_n = 2 b_{n+1}, \, b_n = a_{n+1} + b_{n+1}\} \\
& \cong \lim_{\to} \bigg( \Z^2, \begin{pmatrix} 
0 & 1 \\
2 & 1 
\end{pmatrix} \bigg),
\end{align*}
where each $\Z^2$ is ordered by
$$
(n_1, n_2) \geq 0 \iff n_1, n_2 \geq 0
$$
with an order unit $[1]_0 = 2 a_0 + 4 b_0$.
\end{prop}

From the Remark ~\ref{traceequation}, we see that 
$$
K_0(\tilde{\phi})(a_n) = K_0(\tilde{\phi})(b_n) = 1/ (6 \cdot 2^{n})
$$
for all $n \geq 0$. 
Then the trace map $\tilde{\phi}: C^*(E_{\Z}, \CL, \bE_{\Z}) \to \C$ induces a surjective (not injetive) group homomorphism
$$
K_0(\tilde{\phi}): K_0(C^*(E_{\Z}, \CL, \bE_{\Z})) \to \Q(2^{\infty} \cdot 3)
$$
where $\Q(2^{\infty} \cdot 3)$ is an additive subgroup of $\Q$ consisting of the fractions $x / y$ for $x \in \Z$ and $y = 2^n 3$, $n \geq 1$.
It is easily seen that
$$
K_0(C^*(E_{\Z}, \CL, \bE_{\Z})) / {\rm ker}(K_0(\tilde{\phi})) \cong K_0(C^*(E_{\Z}, \CL, \bE_{\Z})) / \langle a_0 - b_0 \rangle \cong \Q(2^{\infty} \cdot 3).
$$

\section{Representation of $C^*(E_{\Z}, \CL, \bE_{\Z})$ on $\ell^2(\Z)$}\label{6}

\noindent In this section, we introduce a representation of $C^*(E_{\Z}, \CL, \bE_{\Z})$ on the Hilbert space $\ell^2(\Z)$.
Recall that the directed graph $E_{\Z}$ is as follows: 
\vskip 1pc
\hskip 2pc 
\xy /r0.3pc/:(-44.2,0)*+{\cdots};(44.3,0)*+{\cdots};
(-40,0)*+{\bullet}="V-4";
(-30,0)*+{\bullet}="V-3";
(-20,0)*+{\bullet}="V-2";
(-10,0)*+{\bullet}="V-1"; (0,0)*+{\bullet}="V0";
(10,0)*+{\bullet}="V1"; (20,0)*+{\bullet}="V2";
(30,0)*+{\bullet}="V3";
(40,0)*+{\bullet}="V4";
 "V-4";"V-3"**\crv{(-40,0)&(-30,0)};
 ?>*\dir{>}\POS?(.5)*+!D{};
 "V-3";"V-2"**\crv{(-30,0)&(-20,0)};
 ?>*\dir{>}\POS?(.5)*+!D{};
 "V-2";"V-1"**\crv{(-20,0)&(-10,0)};
 ?>*\dir{>}\POS?(.5)*+!D{};
 "V-1";"V0"**\crv{(-10,0)&(0,0)};
 ?>*\dir{>}\POS?(.5)*+!D{};
 "V0";"V1"**\crv{(0,0)&(10,0)};
 ?>*\dir{>}\POS?(.5)*+!D{};
 "V1";"V2"**\crv{(10,0)&(20,0)};
 ?>*\dir{>}\POS?(.5)*+!D{};
 "V2";"V3"**\crv{(20,0)&(30,0)};
 ?>*\dir{>}\POS?(.5)*+!D{};
 "V3";"V4"**\crv{(30,0)&(40,0)};
 ?>*\dir{>}\POS?(.5)*+!D{};
 (-35,1.5)*+{e_{-4}};(-25,1.5)*+{e_{-3}};
 (-15,1.5)*+{e_{-2}};(-5,1.5)*+{e_{-1}};(5,1.5)*+{e_{0}};
 (15,1.5)*+{e_{1}};(25,1.5)*+{e_{2}};(35,1.5)*+{e_{3}};
 (0.1,-2.5)*+{v_0};(10.1,-2.5)*+{v_1};
 (-9.9,-2.5)*+{v_{-1}};
 (-19.9,-2.5)*+{v_{-2}};
 (-29.9,-2.5)*+{v_{-3}};
 (-39.9,-2.5)*+{v_{-4}}; 
 (20.1,-2.5)*+{v_{2}};
 (30.1,-2.5)*+{v_{3}};
 (40.1,-2.5)*+{v_{4}}; 
\endxy 
\vskip 1pc

\noindent For a canonical orthnormal basis $\{v_n : n \in \Z\}$ for $\ell^2(\Z)$, we define shift operators $t_0, t_1$ on $\ell^2(\Z)$ by
$$
t_i (v_n)= \left\{ \begin{array}{ll} v_{n-1}, & \hbox{ if } \quad \CL(e_{n-1}) = i \\
                                                  0 & \hbox{ otherwise.}
                \end{array}
                                \right.  
$$
Then $t_0, t_1$ are partial isometries which have mutually orthogonal ranges, and $t_0 + t_1$ is a bilateral shift on $\ell^2(\Z)$. 
Note that these two shift operators are not periodic.

\begin{prop}
There exists an isomorphism $\Phi : C^*(E_{\Z}, \CL, \bE_{\Z}) \to C^*(t_0, t_1)$ such that $\Phi(s_i) = t_i$, $i = 0, 1$.
\end{prop}

\begin{proof}
For a labeled path $\alpha \in \CL^*(E)$, we define a partial isometry $t_{\alpha}$ by $t_{\alpha_1} t_{\alpha_2} \cdots t_{\alpha_{|\alpha|}}$, and a projection $q_{r(\alpha)}$ by $t_{\alpha}^* t_{\alpha}$.
An easy computation shows that $\{t_{\alpha}, q_{r(\alpha)} : \alpha \in \CL^*(E)\}$ is a representation of $(E_{\Z}, \CL, \bE_{\Z})$. 
Hence there is a $*$-homorphism $\Phi : C^*(E_{\Z}, \CL, \bE_{\Z}) \to C^*(t_0, t_1)$ which maps $s_i \mapsto t_i$.
Obviously $\Phi$ is surjective.
Since $C^*(E_{\Z}, \CL, \bE_{\Z})$ is simple, $\Phi$ is also injective.  
\end{proof}

In general, the above proposition is also true for arbitrary labeled graphs over $E_{\Z}$. 
Let $\CA$ be an (countable) alphabet set and let $\CL : E_{\Z} \to \CA$ be an arbitrary labeling map.
Then the shift operators $\{t_a : a \in \CA\}$ on $\ell^2(\Z)$ given by
$$
t_a (v_n)= \left\{ \begin{array}{ll} v_{n-1}, & \hbox{ if } \quad \CL(e_{n-1}) = a,\\
                                                  0, & \hbox{ otherwise}
                \end{array}
                                \right.  
$$
are well defined, and there exists an isomorphism $\Phi : C^*(E_{\Z}, \CL, \bE) \to C^*(t_a : a \in \CA)$ such that $\Phi(s_a) = t_a$ for $a \in \CA$. 
To prove this, it is enough to show that $\Phi$ is injective.
Note that there exists a strongly continuous action $\lambda : \T \to {\rm Aut}(C^*(t_a : a \in \CA))$ defined by $\lambda_z (t_a) = z t_a$ for $z \in \T$. 
Moreover it is easy to see that $\lambda_z \circ \Phi = \Phi \circ \gamma_z$, where $\gamma$ is a canonical gauge action on $C^*(E_{\Z}, \CL, \bE)$. 
Therefore ``The Gauge-Invariant Uniqueness Theorem (\cite[Theorem 2.7]{BPW})" guarantees that $\Phi$ is injective.

\end{document}